\documentclass[a4paper,12pt]{amsart}
\usepackage{amssymb}
\usepackage{amsmath}
\usepackage{color}

\newcommand{\hide}[1]{}
\newcommand{\ignore}[1]{}

\newcommand{\Z}{\mathbb{Z}}
\newcommand{\Q}{\mathbb{Q}}

\newcommand{\F}{\mathbb{F}}

\newcommand{\leg}[2]{\left(\frac{#1}{#2}\right)}

\newtheorem{dummy}{Dummy}

\newtheorem{lemma}[dummy]{Lemma}
\newtheorem{theorem}[dummy]{Theorem}
\newtheorem*{theorem*}{Theorem}

\newtheorem{cor}[dummy]{Corollary}

\theoremstyle{definition}

\theoremstyle{remark}

\newtheorem*{rem*}{Remark}
\begin{document}
\bibliographystyle{amsalpha}

\author{S.~Mattarei}
\email{smattarei@lincoln.ac.uk}
\address{Charlotte Scott Research Centre for Algebra\\
University of Lincoln \\
Brayford Pool
Lincoln, LN6 7TS\\
United Kingdom}

\author{R.~Tauraso}
\email{tauraso@mat.uniroma2.it}
\urladdr{https://www.mat.uniroma2.it/\~{ }tauraso/}
\address{Dipartimento di Matematica\\
  Universit\`a di Roma ``Tor Vergata'' \\
  via della Ricerca Scientifica\\
  00133 Roma\\
  Italy}

\title[Congruences for partial sums of the generating series for $\binom{3k}{k}$]{Congruences for partial sums\\
of the generating series for $\binom{3k}{k}$}

\begin{abstract}
We produce congruences modulo a prime $p>3$ for sums
$\sum_k\binom{3k}{k}x^k$
over ranges $0\le k<q$ and $0\le k<q/3$, where $q$ is a power of $p$.
Here $x$ equals either $c^2/(1-c)^3$, or $4s^2/\bigl(27(s^2-1)\bigr)$,
where $c$ and $s$ are indeterminates.
In the former case we deal more generally with shifted binomial coefficients $\binom{3k+e}{k}$.
Our method derives such congruences directly from closed forms for the corresponding series.
\end{abstract}


\maketitle

\thispagestyle{empty}

\section{Introduction}

There is a growing literature on
congruences modulo a prime (or sometimes modulo a power of a prime)
for sums involving binomial coefficients.
In several cases such sums are truncated versions of power series
for which a closed form is known.
Similarities of the finite congruences with those closed forms
are often highlighted without making explicit connections.
In~\cite{MatTau:truncation} the authors initiated a systematic
derivation of congruences directly
from closed forms for the corresponding series,
focussing on various sums involving {\em central binomial coefficients}
$\binom{2k}{k}$, or the related {\em Catalan numbers} $C_k=(k+1)^{-1}\binom{2k}{k}$.
In that case the paradigm was the congruence
$
\sum_{0\le k<q}\binom{2k}{k}x^k\equiv(1-4x)^{(q-1)/2}\pmod{p},
$
which is not hard to prove directly but may conveniently be deduced from the well-known identity
%\begin{equation}\label{eq:central_gf}
$
\sum_{k=0}^\infty\binom{2k}{k}x^k
=
(1-4x)^{-1/2}
% 1/\sqrt{1-4x}.
$
%\end{equation}
via a procedure that one may call {\em truncation and reduction modulo $p$.}
A range of variations was systematically investigated, and substitution
of rational, or more generally algebraic numbers, for $x$ yielded various interesting numerical congruences, such as
$
\sum_{0\le k<p}\binom{2k}{k}k^{-3}\equiv 2B_{p-3}/3\pmod{p},
$
where $p>3$ is a prime and $B_{p-3}$ is a Bernoulli number.

In this paper we investigate certain sums
involving binomial coefficients of the form $\binom{3k}{k}$.
% , or their shifted version $\binom{3k+e}{k}$.
More generally, one may consider
the power series $y=\sum_{k=0}^\infty\binom{rk}{k}x^k$.
Because that series satisfies
$(y-1)\bigl((r-1)y+1\bigr)^{r-1}-r^rxy^r=0$,
an equation of degree $r$ in $y$
(see Equation~\eqref{eq:y_equation} below),
the existence of a closed form for the series depends on being able to `solve' that equation.
When $r=3$, Cardano's formula yields a
closed form for $y$,
to which one may then apply the machinery of
truncation and reduction modulo $p$ and obtain corresponding congruences
for the truncated sums.
We carry out that in Section~\ref{sec:Cardano}, in terms of an accessory indeterminate $s$
in place of $x$, where $x=4s^2/\bigl(27(s^2-1)\bigr)$.
That substitution has the simplifying effect of turning the discriminant
of the cubic equation into a perfect square.
By evaluating the resulting congruence at rational values of the indeterminate,
or even irrational but $p$-integral algebraic values, we discover interesting numerical congruences such as
$
\sum_{0\leq k<q/3}\binom{3k}{k}
3^{-k}
\equiv
\varepsilon F_{2(2q+\varepsilon)/3}
\pmod{p}
$
and
$
\sum_{q/2<k<2q/3}\binom{3k}{k}
3^{-k}
\equiv
\varepsilon F_{2(q-\varepsilon)/3}
\pmod{p},
$
in terms of Fibonacci numbers,
where $p>3$ and $\varepsilon=\leg{q}{3}$ denotes a Legendre symbol.
We provide a wider sample of such numerical congruences in Section~\ref{sec:numerical}.

An alternate approach to solving the above-mentioned equation of degree $r$
for the series $y$ is the possibility of parametrizing one special solution
of the equation, different from the one we are interested in, thus allowing
the left-hand side of the equation to factorize, with our series $y$ being
a root of the remaining factor of degree $r-1$.
The details of this procedure are explained in Section~\ref{sec:series},
and are carried out in terms of the more general series
$\sum_{k=0}^\infty\binom{rk+e}{k}x^k$, where $e$ is a nonnegative integer.
Note that treating shifted versions $\binom{rk+e}{k}$ is
more general than restricting to shifts of the form $\binom{rk}{k-d}$
as done in some papers, because the latter can be written as
$\binom{rh+rd}{h}$ with $h=k-d$.

When $r=3$ this allows the series, once written in terms of an accessory indeterminate
$c$, where $x=c^2/(1-c)^3$, to have a closed form involving only one square root
extraction, which is Equation~\eqref{eq:3k+e_series} below.
A further accessory indeterminate $\beta$, related to $c$ by $c=\beta(1-\beta)$,
allows one to avoid explicit square root extraction and express the closed form
as a rational function of $\beta$.
This device, which was already employed in~\cite{MatTau:truncation},
facilitates the subsequent truncation process.
Our main result here is Theorem~\ref{thm:3k+e_su_k},
in Section~\ref{sec:congruences}, which states congruences
for certain finite sums
$\sum\binom{3k+e}{k}x^k$
in terms of rational functions of $\beta$.
The natural finite summation range $0\le k<q$
for those sums decomposes further into to three natural subintervals
according to Lucas' theorem.
The proof of Theorem~\ref{thm:3k+e_su_k} is the longest in this paper and occupies Section~\ref{section:proof_3k+e_su_k}.

In Section~\ref{sec:exploiting} we present some applications of Corollary~\ref{cor:3k_su_k},
which is the special case $e=0$ of Theorem~\ref{thm:3k+e_su_k}, and as such has a simpler formulation.
In particular, Theorem~\ref{thm:zero} characterizes the values of $a\in\F_q$,
the field of $q$ elements, such that
$\sum_{0<k<q/3}\binom{3k}{k}a^k\equiv 0\pmod{p}$.

\section{The power series $\sum_{k=0}^{\infty}\binom{rk+e}{k}x^k$}\label{sec:series}

In this section we collect some information on the generating function of the binomial coefficients $\binom{rk+e}{k}$ as a function of $k$.
For $r$ a positive integer, the power series
%$\mathcal{B}_r(x)\in\Q[[x]]$ given by
\begin{equation}\label{eq:B_r_def}
\mathcal{B}_r(x)
=\sum_{k=0}^{\infty}\frac{1}{rk+1}\binom{rk+1}{k}x^k
=\sum_{k=0}^{\infty}\frac{1}{(r-1)k+1}\binom{rk}{k}x^k
\end{equation}
was called the {\em generalized binomial series} in~\cite[Equation~(5.58)]{GKP}.
Note that $\mathcal{B}_1(x)=1/(1-x)$.
According to~\cite[Example~6.2.6]{Stanley:EC2},
$\mathcal{B}_r(x)$ satisfies
\begin{equation}\label{eq:B_r_equation}
\mathcal{B}_r(x)=1+x\mathcal{B}_r(x)^r,
\end{equation}
which can be proved using Lagrange inversion.
More generally, for $e>0$ Lagrange inversion
produces
\begin{equation}\label{eq:GKP(5.60)}
\mathcal{B}_r(x)^e
=\sum_{k=0}^{\infty}\frac{e}{rk+e}\binom{rk+e}{k}x^k,
\end{equation}
which is~\cite[Equation~(5.60)]{GKP}.
One may  also obtain Equation~\eqref{eq:GKP(5.60)}
inductively from Equation~\eqref{eq:B_r_def}
using the Rothe-Hagen convolution identity~\cite[Equation~(5.63)]{GKP}.
The series in Equation~\eqref{eq:GKP(5.60)} is the ordinary generating function of the Fuss-Catalan numbers, a generalization of the Catalan numbers introduced by Nicolaus Fuss in the late eighteenth century.
%in 1791. OR 1795???
Differentiating Equations~\eqref{eq:B_r_equation} and~\eqref{eq:GKP(5.60)},
and then eliminating the derivative of $\mathcal{B}_r(x)$, one finds
\begin{equation}\label{eq:GKP(5.61)}
\frac{\mathcal{B}_r(x)^e}{1-r+r\mathcal{B}_r(x)^{-1}}
=\sum_{k=0}^{\infty}\binom{rk+e}{k}x^k,
\end{equation}
which is~\cite[Equation~(5.61)]{GKP}.
Although this derivation is only valid for $e>0$,
Equation~\eqref{eq:GKP(5.61)} holds for $e=0$ as well,
as one can see by differentiating the second expression for $\mathcal{B}_r(x)$ given in
Equation~\eqref{eq:B_r_def} instead of Equation~\eqref{eq:GKP(5.60)}.

Equation~\eqref{eq:GKP(5.61)} shows that each formal power series $y_{r,e}(x)=\sum_{k=0}^{\infty}\binom{rk+e}{k}x^k\in\Q[[x]]$ is algebraic,
because so is $\mathcal{B}_r(x)$ according to Equation~\eqref{eq:B_r_def}.
This means that $y_{r,e}(x)$ belongs to a finite-degree extension field of the field $\Q((x))$ of formal Laurent series.
In fact, $\mathcal{B}_r(x)$ is algebraic of degree $r$, with minimal polynomial $z^r-zx^{-1}+x^{-1}$ obtained from Equation~\eqref{eq:B_r_def}.
(That is indeed the minimal polynomial because it is irreducible over $\Q((x))$.)
Since $y_{r,e}(x)$ belongs to the extension field of $\Q((x))$ of $\Q((x))$ generated by $\mathcal{B}_r(x)$, is also algebraic, of degree not exceeding $r$.
It is not hard to show that $y_{r,e}(x)$ has degree precisely $r$.
Consequently, $y_{r,e}(x)$ satisfies an equation of degree $r$ analogous to Equation~\eqref{eq:B_r_def}.
Such an equation is awkward when worked out in general, and we will have no need for that in this paper, except for the special case $e=0$,
which is easy to deduce from Equation~\eqref{eq:GKP(5.61)} and Equation~\eqref{eq:B_r_equation}:
the power series $y=y_{r,0}(x)=\sum_{k=0}^\infty\binom{rk}{k}x^k$ satisfies
\begin{equation}\label{eq:y_equation}
(y-1)\bigl((r-1)y+1\bigr)^{r-1}-r^rxy^r=0.
\end{equation}
This equation can also be found in~\cite[Example~6.2.7 ]{Stanley:EC2}.

In principle, a closed form for the series $y_{r,e}(x)$ in terms of radicals and rational expressions may be obtained for $r\le 4$
by solving the corresponding equation of degree $r$ using radicals.
This is straightforward for $r=2$ and leads to familiar closed forms.
For $r=3$ one may use Cardano's formula, but that is more easily done through an artifice which renders the discriminant
(almost) a perfect square, and we devote Section~\ref{sec:Cardano} to that approach in the special case $e=0$.

Here we discuss a different artifice, which allows one to pass from degree $r$ to one less in the general case.
In order to characterize $\mathcal{B}_r(x)$ among the roots of Equation~\eqref{eq:B_r_equation}, it is more convenient to work with
its reciprocal.
The power series
$w=w(x)=1/\mathcal{B}_r(x)$
is the only solution of the equation
$w^r-w^{r-1}+x=0$
such that $w(0)=1$.
If we set $x=-c^{r-1}/(c-1)^r$,
then the resulting equation has $w=c/(c-1)$
among its roots, and its left-hand side factorizes as
\begin{align*}
w^r-w^{r-1}-\frac{c^{r-1}}{(c-1)^r}
=
\left(w-\frac{c}{c-1}\right)
\left(w^{r-1}+\sum_{i=0}^{r-2}\frac{c^i}{(c-1)^{i+1}}w^{r-2-i}\right).
\end{align*}
Consequently, the series
$w=1/\mathcal{B}_r\bigl(-c^{r-1}/(c-1)^r\bigr)$
is the only solution of the equation
\begin{equation}\label{eq:w_equation}
w^{r-1}+\sum_{i=0}^{r-2}\frac{c^i}{(c-1)^{i+1}}w^{r-2-i}=0
\end{equation}
satisfying $w(0)=1$.
Our gain in passing from the indeterminate $x$ to $c$ lies in this equation having degree one less than the original equation $w^r-w^{r-1}+x=0$.

In particular, when $r=2$ Equation~\eqref{eq:w_equation} reads $w+1/(c-1)$, and hence
$\mathcal{B}_2\bigl(-c/(c-1)^2\bigr)=1-c$.
Equation~\eqref{eq:GKP(5.61)} then gives us
\[
\sum_{k=0}^{\infty}\binom{2k+e}{k}\left(\frac{-c}{(c-1)^2}\right)^k=\frac{(1-c)^{e+1}}{1+c}.
\]
Here $c$ can easily be obtained from $x$, as
$c=1-(1-\sqrt{1-4x})/(2x)$,
which leads to the better-known equation
\[
\sum_{k=0}^{\infty}\binom{2k+e}{k}x^k=
\frac{1}{\sqrt{1-4x}}
\left(\frac{1-\sqrt{1-4x}}{2x}\right)^e,
\]
see~\cite[Equation~(2.47)]{Wilf}.

When $r=3$ we find that
$w=1/\mathcal{B}_3\bigl(-c^{2}/(c-1)^3\bigr)$
is the only solution of the equation
\[
w^{2}+\frac{1}{c-1}w+\frac{c}{(c-1)^2}=0
\]
such that $w(0)=1$.
Hence one obtains
\[
\mathcal{B}_3\left(c^{2}/(1-c)^3\right)=(1-c)\frac{1-\sqrt{1-4c}}{2c}.
\]
It is now convenient to set
$\beta=(1-\sqrt{1-4c})/2$.
Noting that $\beta(1-\beta)=c$ we find
\[
\mathcal{B}_3\left(c^{2}/(1-c)^3\right)=\frac{1-\beta+\beta^2}{1-\beta}
=\frac{1+\beta^3}{1-\beta^2}.
\]
Equation~\eqref{eq:GKP(5.61)} then gives us
%The second equation of~\cite{GKP} then gives us
\begin{equation}\label{eq:3k+e_series}
\sum_{k=0}^{\infty}\binom{3k+e}{k}\left(\frac{c^2}{(1-c)^3}\right)^k
=
\frac{1}{(1+\beta)(1-2\beta)}
\frac{(1-\beta+\beta^2)^{e+1}}{(1-\beta)^e}.
\end{equation}
In the next sections we will derive from this equation a congruence modulo a prime $p$ for certain finite sums,
obtained by truncating the series at appropriate places.
For comparison, with the same notation we have
\[
\sum_{k=0}^{\infty}\binom{2k+e}{k}c^k
=
\frac{1}{(1-2\beta)(1-\beta)^e},
\]
which was used as a starting point for deducing congruences in the proof of~\cite[Theorem~5]{MatTau:truncation}.

\section{Congruences for finite sums $\sum_{k}\binom{rk+e}{k}x^k$ modulo a prime}\label{sec:congruences}

Our first goal in this paper is an evaluation, in closed form and as polynomial congruences modulo a prime,
of finite sums $\sum_{k}\binom{3k+e}{k}x^k$ over certain ranges.
We start with describing certain natural ranges for evaluations modulo a prime coming from Lucas' theorem, for the more general
sums $\sum_{k}\binom{rk+e}{k}x^k$,
which are refinements of the basic natural range $0\le k<q$, where $q$ is a power of $p$.

\begin{lemma}\label{lemma:binomial_vanish}
Let $r$ be a positive integer, let $q$ be a power
of a prime $p$, and let $0\le e<q$.
Then the binomial coefficient $\binom{rk+e}{k}$ for $0\le k<q$
is a multiple of $p$ unless
$k\in A(r,m,e)$
for some $0<m\le r$, where
\[
A(r,m,e)=\left\{k\in\Z:
\frac{(m-1)q-e}{r-1}
\le k<
\frac{mq-e}{r}
\right\}.
\]
\end{lemma}

\begin{proof}
Because
$\binom{rk+e}{k}\equiv\binom{rk+e-mq}{k}\pmod{p}$
for any integer $m$ according to Lucas' Theorem,
$\binom{rk+e}{k}\equiv 0\pmod{p}$
holds if and only if $0\le rk+e-mq<k$, which means
$(mq-e)/r\le k<(mq-e)/(r-1)$.
These are the complementary intervals to the intervals $A(r,m,e)$
within the range $0\le k<q$.
\end{proof}

Thus, when considering finite sums $\sum_k\binom{rk+e}{k}x^k$ modulo a prime $p$,
and $q$ is any power of $p$,
the range $0\le k<q$ splits naturally into $r$ separate ranges,
possibly including empty ones such as $A(r,r,0)$.
% with the last one (for $m=r$) being empty if $e=0$.
Consequently, it is natural to look for evaluations modulo $p$ of
the partial sums
\[
\sum_{0\le k<(mq-e)/r}\binom{rk+e}{k}x^k,
\]
for $0<m\le r$,
or on the subintervals $A(r,m,e)$ in which this range decomposes naturally according to Lemma~\ref{lemma:binomial_vanish}.

When $r=2$ the ranges of Lemma~\ref{lemma:binomial_vanish} read
$0\le k<(q-e)/2$
and
$q-e\le k<q-e/2$.
Finite sums $\sum_k\binom{2k+e}{k}x^k$ over each of those two intervals were evaluated, in closed form modulo $p$,
in~\cite[Theorem~45]{MatTau:truncation}.
Because we will rely on that result to deal with the case $r=3$,
and because the latter will require a slightly different approach,
we provide a new proof of~\cite[Theorem~45]{MatTau:truncation}
by way of introduction to our new approach.
The main novelty is that we can prove the desired congruence over the first interval $0\le k<(q-e)/2$
without having to consider both intervals together, as we did in the original proof.
Here we prefer to use the letter $c$ for the indeterminate in place of $x$,
because the former bears the same relationship to the indeterminate $\beta$
as that in place when we will deal with sums $\sum_k\binom{3k+e}{k}x^k$ later.

\begin{theorem}[Part of Theorem~45 of~\cite{MatTau:truncation}]\label{thm:2k+e_su_k}
Let $q$ be a power of an odd prime $p$, let $1\le m\le 2$, and let $0\le e\le q$.
In the polynomial ring $\Z[\beta]$, setting $c=\beta(1-\beta)$ and $\alpha=1-\beta$, we have
\begin{equation*}
\sum_{0\le k<(mq-e)/2}\binom{2k+e}{k}c^k
\equiv
\frac{\alpha^{mq-e}-\beta^{mq-e}}{\alpha-\beta}
\pmod{p}.
\end{equation*}
\end{theorem}

Although the right-hand side of the congruence does not look like a polynomial in $\beta$, it reduces to one after simplification.

\begin{proof}
We will first prove the case $m=1$, and then deduce the case $n=2$ from that.
We start from the identity
\[
\sum_{k=0}^\infty\binom{2k+e}{k}c^k
=
\frac{1}{(1-2\beta)(1-\beta)^e},
\]
which takes place in the power series ring $\Q[[\beta]]$, where $c=\beta(1-\beta)$.
However, because all coefficients are integers it actually takes place in $\Z[[\beta]]$.
After multiplying both sides by $(1-\beta)^e$ and then by $(1-2\beta)^q\equiv 1\pmod{(\beta^q,p)}$ we obtain
\[
\sum_{0\le k<(q-e)/2}\binom{2k+e}{k}c^k
\equiv\frac{1}{(1-2\beta)(1-\beta)^e}
\pmod{(\beta^{q-e},p)}
\]
in $\Z[[\beta]]$.
In fact, $\binom{2k+e}{k}\equiv 0\pmod{p}$ for $(q-e)/2\le k<q-e$
according to Lemma~\ref{lemma:binomial_vanish},
and
$\sum_{k\ge q-e}\binom{2k+e}{k}c^k\equiv 0\pmod{(\beta^{q-e},p)}$.
We also have
\[
\frac{(1-\beta)^{q-e}-\beta^{q-e}}{1-2\beta}
\equiv
\frac{1}{(1-2\beta)(1-\beta)^e}
\pmod{(\beta^{q-e},p)}.
\]
The left-hand sides of the previous two congruences
are polynomials of degree less than $q-e$, and hence so is their difference.
However, when the difference is viewed as a polynomial in $\F_p[\beta]$,
we have just shown that it is a multiple of $\beta^{q-e}$.
Consequently the difference must be zero in $\F_p[\beta]$,
and the desired conclusion follows.

Now we may deduce the case $m=2$ from the case $m=1$.
Using Lucas'theorem and the basic
binomial coefficient identity $\binom{n}{k}=\binom{n}{n-k}$ we find
\begin{align*}
\sum_{q-e\le k<q-e/2}\binom{2k+e}{k}c^k
&=
c^{q-e}
\sum_{0\le k<e/2}\binom{2k+2q-e}{k+q-e}c^k
\\&\equiv
c^{q-e}
\sum_{0\le k<e/2}\binom{2k+q-e}{k+q-e}c^k
\pmod{p}
\\&=
c^{q-e}
\sum_{0\le k<e/2}\binom{2k+q-e}{k}c^k.
\end{align*}
Now the case $m=1$ with $q-e$ in place of $e$ yields
\[
\sum_{q-e\le k<q-e/2}\binom{2k+e}{k}c^k
\equiv
\frac{\alpha^q\beta^{q-e}-\alpha^{q-e}\beta^q}{\alpha-\beta}
\pmod{p},
\]
and adding this to the sum over the range $0\le k<(q-e)/2$
we easily reach the desired conclusion.
\end{proof}

After having reviewed the case $r=2$, we move on to the case $r=3$, which is the one of main interest in this paper.
According to Lemma~\ref{lemma:binomial_vanish},
we are interested in evaluating sums $\sum_{k}\binom{3k+e}{k}x^k$ modulo $p$, for $0\le e<q$,
over each of the three finite ranges
\[
0\le k<(q-e)/3,
\quad
(q-e)/2\le k<(2q-e)/3,
\quad
q-e/2\le k<q-e/3.
\]

\begin{theorem}\label{thm:3k+e_su_k}
Let $q$ be a power of an odd prime $p$, let $1\le m\le 3$,
and let $0\le e <q$.
In the polynomial ring $\Z[\beta]$, setting $c=\beta(1-\beta)$, $\alpha=1-\beta$, and $x=c^2/(1-c)^3$, we have
\begin{align*}
2&(2+c)(1-c)^{mq-1-e}
\sum_{0\le k<(mq-e)/3}\binom{3k+e}{k}
x^k
\\&\equiv
(\alpha^{mq-e}+\beta^{mq-e})
+3\frac{\alpha^{mq-e}-\beta^{mq-e}}{\alpha-\beta}
-2(-c)^{mq-e}
\pmod{p}.
\end{align*}
\end{theorem}

We explicitly state the special case $e=0$ as a corollary, because
the formulas then simplify and take place in the polynomial ring
$\Z[c]$, without the explicit involvement of the
indeterminate $\beta$.

\begin{cor}\label{cor:3k_su_k}
For any power $q$ of an odd prime $p$,
in the polynomial ring $\Z[c]$, where $x=c^2/(1-c)^3$, we have
\[
2(2+c)
(1-c)^{q-1}
\sum_{0\le k<q/3}\binom{3k}{k}x^k
\equiv
1+3(1-4c)^{(q-1)/2}
+2c^q
\pmod{p},
\]
and
\[
2(2+c)
(1-c)^{2q-1}
\sum_{0\le k< 2q/3}\binom{3k}{k}x^k
\equiv
1+3(1-4c)^{(q-1)/2}
-2c^q-2c^{2q}
\pmod{p}.
\]
\end{cor}

\begin{proof}
When $e=0$, for $m=1$ the right-hand side of the congruence of
Theorem~\ref{thm:3k+e_su_k} reads
\[
(\alpha^{q}+\beta^{q})
+3\frac{\alpha^{q}-\beta^{q}}{\alpha-\beta}
-2(-c)^{q}
\equiv
1
+3(\alpha-\beta)^{q-1}
+2c^{q}
\pmod{p},
\]
and the conclusion follows because
$(\alpha-\beta)^2=(1-2\beta)^2=1-4\beta+4\beta^2=1-4c$.
For $m=2$ the desired conclusion follows similarly because
$\alpha^{2q}+\beta^{2q}=(\alpha^{q}+\beta^{q})^2-2\alpha^q\beta^q
\equiv 1-2c^q\pmod{p}$
and
$\alpha^{2q}-\beta^{2q}=(\alpha^{q}+\beta^{q})(\alpha^{q}-\beta^{q})\equiv(\alpha-\beta)^q\pmod{p}$.
Of course when $e=0$ we do not get anything new for $m=3$.
\end{proof}

According to Corollary~\ref{cor:3k_su_k} the sums over the two ranges are related by the congruence
\[
\sum_{0\le k< q/3}\binom{3k}{k}x^k
-(1-c^q)\sum_{0\le k< 2q/3}\binom{3k}{k}x^k
\equiv
c^q\frac{(2+c)^{q-1}}{(1-c)^{q-1}}
\pmod{p}.
\]

Theorem~\ref{thm:3k+e_su_k} and Corollary~\ref{cor:3k_su_k}
remain trivially valid also when $p=2$, but provide no information
on the corresponding sums.
According to Lucas' theorem, the binomial coefficient $\binom{3k}{k}$ is odd
precisely when the binary expansion of $k$
contains no adjacent digits equal to $1$.
%Such integers $k$ are called {\em Fibbinary numbers}.
A well-known combinatorial characterization of the Fibonacci numbers then implies
$\sum_{0\le k<2^r}\binom{3k}{k}\equiv F_{r+2}\pmod{2}$.
%where $F_{r+2}$ is a Fibonacci number.
We will not pursue the case $p=2$ further in this paper.

\section{Proof of Theorem~\ref{thm:3k+e_su_k}}\label{section:proof_3k+e_su_k}

We will deduce the desired congruences from the closed form
for the corresponding series, which we gave in Equation~\eqref{eq:3k+e_series}.
Because $1-\beta+\beta^2=1-c$ and $(2-\beta)(1+\beta)=2+c$
we may rewrite that identity in the form
\begin{equation*}
\sum_{k=0}^{\infty}\binom{3k+e}{k}\left(\frac{c^2}{(1-c)^3}\right)^k
=
\frac{1-c}{2(2+c)}
\left(1+\frac{3}{1-2\beta}\right)
\frac{(1-c)^e}{(1-\beta)^e}.
\end{equation*}
%where $x=c^2/(1-c)^3=\beta^2(1-\beta)^2/(1-\beta+\beta^2)^3$.

We start with the case $m=1$.
In order to clear denominators of the left-hand side of the above identity
in the first range $0\le k<q/3$ that we are interested in,
we multiply both sides by $(1-c)^{q-1-e}$.
After further multiplying both sides by $2(2+c)$ we find
\begin{equation}\label{eq:3q+e_identity}
2(2+c)
\sum_{k=0}^{\infty}\binom{3k+e}{k}
c^{2k}(1-c)^{q-1-e-3k}
=
\frac{(1-c)^q}{(1-\beta)^e}
\left(1+\frac{3}{1-2\beta}\right),
\end{equation}
to be viewed as an identity in the power series ring $\Q[[\beta]]$, and actually $\Z_{(p)}[[\beta]]$ (so we can view it modulo $p$).
Now we produce congruences, in turn, for each side of Equation~\eqref{eq:3q+e_identity}.

Because the binomial coefficient $\binom{3k+e}{k}$
is a multiple of $p$
for $(q-e)/3\le k<(q-e)/2$,
the left-hand side of Equation~\eqref{eq:3q+e_identity} satisfies
\begin{equation}\label{eq:LHS}
\begin{aligned}
2&(2+c)
\sum_{k=0}^{\infty}\binom{3k+e}{k}
c^{2k}(1-c)^{q-1-e-3k}
\\&\equiv
2(2+c)
\sum_{0\le k<(q-e)/3}\binom{3k+e}{k}
c^{2k}(1-c)^{q-1-e-3k}
\pmod{(c^{q-e},p)}.
\end{aligned}
\end{equation}
The right-hand side of this congruence
is a polynomial in
$c$, of degree $q-e$ and leading term $-2(-c)^{q-e}$.

Before we consider the right-hand side of
Equation~\eqref{eq:3q+e_identity}, note that
for $m\in\{1,2,3\}$ we have
\[
1-mc^q
=1-m\beta^q+m\beta^{2q}
\equiv(1-\beta^q)^m
\equiv\alpha^{mq}
%\equiv\alpha^{mq}+\beta^{mq}
\pmod{(\beta^{mq},p)},
\]
where we have set $\alpha=1-\beta$.
Consequently,
\begin{equation}\label{eq:1-mc^q}
\begin{aligned}
\frac{1-mc^q}{(1-\beta)^e}
&\equiv
\alpha^{mq-e}\pm\beta^{mq-e}
\pmod{(\beta^{mq-e},p)}.
\end{aligned}
\end{equation}
In particular, the right-hand side of
Equation~\eqref{eq:3q+e_identity} satisfies
\begin{equation}\label{eq:1-c^q}
\frac{(1-c)^q}{(1-\beta)^e}
\left(1+\frac{3}{1-2\beta}\right)
\equiv
(\alpha^{q-e}+\beta^{q-e})
+3\frac{\alpha^{q-e}-\beta^{q-e}}{\alpha-\beta}
\pmod{(\beta^{q-e},p)}.
\end{equation}

Combining Equations~\eqref{eq:LHS} and ~\eqref{eq:1-c^q} we obtain
\begin{equation}\label{eq:LHS=RHS}
\begin{aligned}
2&(2+c)
\sum_{0\le k<(q-e)/3}\binom{3k+e}{k}
c^{2k}(1-c)^{q-1-e-3k}
\\&
\equiv
(\alpha^{q-e}+\beta^{q-e})
+3\frac{\alpha^{q-e}-\beta^{q-e}}{\alpha-\beta}
\pmod{(\beta^{q-e},p)}.
\end{aligned}
\end{equation}
The right-hand side of this congruence is invariant under
interchanging $\beta$ with $\alpha=1-\beta$, and hence
can be written as a polynomial in their elementary symmetric polynomials
$\alpha+\beta=1$ and $\alpha\beta=c$.
Hence the right-hand side of Equation~\eqref{eq:LHS=RHS}
is actually a polynomial in $c=\beta(1-\beta)$.
% as well as the left-hand side, of course.
Because $\beta$ and $1-\beta$ are coprime,
it follows that the congruence actually holds modulo $(c^{q-e},p)$.
Also, because the right-hand side of Equation~\eqref{eq:LHS=RHS}
has degree at most $q-e$ as a polynomial in $\beta$,
it has degree at most $(q-e)/2$ as a polynomial in $c$,
and hence less than $q-e$.
The desired congruence modulo $p$ follows because
the left-hand side of Equation~\eqref{eq:LHS=RHS}
has leading term  $-2(-c)^{q-e}$, as noted earlier.

Now we deal with the case $m=2$,
where the finite sum is over the range $0\le k<(2q-e)/3$.
We proceed in a similar fashion,
but in order to clear denominators over the longer range we first
need to multiply both sides of Equation~\eqref{eq:3q+e_identity} by
a further factor $(1-c)^q$.
Because $\binom{3k+e}{k}$ is a multiple of $p$
for $(2q-e)/3\le k<(2q-e)/2$,
the left-hand side of Equation~\eqref{eq:3q+e_identity} multiplied by $(1-c)^q$
satisfies
\begin{equation}\label{eq:LHS_midrange}
\begin{aligned}
2&(2+c)
\sum_{k=0}^{\infty}\binom{3k+e}{k}
c^{2k}(1-c)^{2q-1-e-3k}
\\&\equiv
2(2+c)
\sum_{0\le k<(2q-e)/3}\binom{3k+e}{k}
c^{2k}(1-c)^{2q-1-e-3k}
\pmod{(c^{2q-e},p)}.
\end{aligned}
\end{equation}
As a polynomial in $c$ the right-hand side of this congruence
has degree $2q-e$ and leading term
$2(-c)^{2q-e}$.

The right-hand side of Equation~\eqref{eq:3q+e_identity}
also needs to be multiplied by $(1-c)^q$, and then the result contains the factor
$(1-c)^{2q}\equiv 1-2c^q\pmod{c^{2q}}$.
Using Equation~\eqref{eq:1-mc^q} for $m=2$
we find that the right-hand side of
Equation~\eqref{eq:3q+e_identity} multiplied by $(1-c)^q$
satisfies
\[
\frac{(1-c)^{2q}}{(1-\beta)^e}
\left(1+\frac{3}{1-2\beta}\right)
\equiv
(\alpha^{2q-e}+\beta^{2q-e})
+3\frac{\alpha^{2q-e}-\beta^{2q-e}}{\alpha-\beta}
\pmod{(\beta^{2q-e},p)}.
\]
Combining this congruence with Equation~\eqref{eq:LHS_midrange}
we find a version of the desired conclusion as a congruence modulo
$(\beta^{2q-e},p)$.
Arguing as we did for the case $m=1$, we observe how symmetry makes
the congruence hold modulo $(c^{2q-e},p)$.
Finally, keeping track of the leading term
we obtain the desired conclusion for $m=2$.

To deal with the final case $m=3$,
where the finite sum is over the range $0\le k<(3q-e)/3$,
we cannot proceed exactly in the same way as we have just done for $m=1,2$.
% thus multiplying the left-hand side of
% Equation~\eqref{eq:3q+e_identity} by $(1-c)^{2q}$.
In fact, a congruence analogous to Equation~\eqref{eq:LHS_midrange},
with both sides multiplied by a further factor
$(1-c)^q$, and the summation at the right-hand side
extended to $0\le k<(3q-e)/3$, does not hold modulo $(c^{3q-e},p)$
as we would need to carry out a similar argument,
but only modulo $(c^{2q},p)$.
That is because
$\binom{3k+e}{k}$ is not a multiple of $p$
for $(3q-e)/3\le k<(3q-e)/2$,
but only on the shorter range $(3q-e)/3\le k<q$.

To overcome this obstacle we evaluate a longer partial sum,
over the range $0\le k<(4q-e)/3=q+(q-e)/3$,
of the left-hand side of Equation~\eqref{eq:3q+e_identity}
multiplied by $(1-c)^{3q}$.
According to Lucas' theorem,
for $q\le k<(4q-e)/3$ we have
\[
\binom{3k+e}{k}\equiv\binom{3q}{q}\binom{3(k-q)+e}{k-q}
\equiv 3\binom{3(k-q)+e}{k-q}\pmod{p},
\]
and for $(4q-e)/3\le k<(3q-e)/2$ we have
\[
\binom{3k+e}{k}\equiv\binom{4q}{q}\binom{3k-4q+e}{k-q}
\equiv 0\pmod{p}.
\]
Consequently, splitting the summation range
$0\le k<(4q-e)/3$
into two portions
$0\le k<(3q-e)/3$
and
$q\le k<(4q-e)/3=q+(q-e)/3$
(with the range $(3q-e)/3\le k<q$ between them giving no
contribution according to Lemma~\ref{lemma:binomial_vanish}), we find
\begin{equation}\label{eq:LHS_fullrange}
\begin{aligned}
2&(2+c)
\sum_{k=0}^{\infty}\binom{3k+e}{k}
c^{2k}(1-c)^{4q-1-e-3k}
\\&\equiv
2(2+c)
\sum_{0\le k<(4q-e)/3}\binom{3k+e}{k}
c^{2k}(1-c)^{4q-1-e-3k}
\pmod{(c^{3q-e},p)}
\\&\equiv
(1-c)^q 2(2+c)\sum_{0\le k<(3q-e)/3}\binom{3k+e}{k}
c^{2k}(1-c)^{3q-1-e-3k}
\\&
\quad+ 3c^{2q}2(2+c)\sum_{0\le k<(q-e)/3}\binom{3k+e}{k}
c^{2k}(1-c)^{q-1-e-3k}\pmod{p}.
\end{aligned}
\end{equation}

The right-hand side of Equation~\eqref{eq:3q+e_identity}
also needs to be multiplied by $(1-c)^{3q}$, and then the result contains the factor
$(1-c)^{4q}\equiv (1-c^q)(1-3c^q)+3(1-c^q)c^{2q}\pmod{c^{3q}}$.
Using Equation~\eqref{eq:1-mc^q} for $m=3$,
and Equation~\eqref{eq:1-c^q}, we find
\[
\begin{aligned}
\frac{(1-c)^{4q}}{(1-\beta)^e}&
\left(1+\frac{3}{1-2\beta}\right)
\\&\equiv
(1-c)^q\left(
(\alpha^{3q-e}+\beta^{3q-e})
+3\frac{\alpha^{3q-e}-\beta^{3q-e}}{\alpha-\beta}\right)\\
&\quad+
3c^{2q}\left(
(\alpha^{q-e}+\beta^{q-e})
+3\frac{\alpha^{q-e}-\beta^{q-e}}{\alpha-\beta}\right)
\pmod{(\beta^{3q-e},p)}.
\end{aligned}
\]
Using our conclusion in the case $m=1$ we find
\begin{align*}
(1-c)^q &2(2+c)\sum_{0\le k<(3q-e)/3}\binom{3k+e}{k}
c^{2k}(1-c)^{3q-1-e-3k}
\\&\equiv
(1-c)^q\left(
(\alpha^{3q-e}+\beta^{3q-e})
+3\frac{\alpha^{3q-e}-\beta^{3q-e}}{\alpha-\beta}\right)
\pmod{(\beta^{3q-e},p)}.
\end{align*}
Because the factor $(1-c)^q$ is coprime with the modulus $\beta^{3q-e}$,
we deduce
\begin{align*}
2&(2+c)\sum_{0\le k<(3q-e)/3}\binom{3k+e}{k}
c^{2k}(1-c)^{3q-1-e-3k}
\\&\equiv
(\alpha^{3q-e}+\beta^{3q-e})
+3\frac{\alpha^{3q-e}-\beta^{3q-e}}{\alpha-\beta}
\pmod{(\beta^{3q-e},p)}.
\end{align*}
Arguing as we did in previous cases, the right-hand side is actually
a polynomial in $c$, and hence the congruence holds modulo $(c^{3q-e},p)$.
As a polynomial in $c$ the right-hand side has degree less than $3q-e$,
and after accounting for the leading term of the left-hand side, which is $2(-c)^{3q-e}$, we obtain the desired conclusion for $m=3$.

The proof of Theorem~\ref{thm:3k+e_su_k} is now complete.

\section{Exploiting polynomial congruences}\label{sec:exploiting}

Working modulo $c^q-c$, and conveniently separating
the initial term of the summation in the congruences of
Corollary~\ref{cor:3k_su_k},
we deduce the weaker but simpler congruences
\begin{equation}\label{eq:ZWS_short}
2(2+c)
\sum_{0<k<q/3}\binom{3k}{k}
x^k
\equiv
-3+3(1-4c)^{(q-1)/2}
\pmod{(c^q-c,p)},
\end{equation}
and
\begin{equation}\label{eq:ZWS}
2(2+c)(1-c)
\sum_{0<k<q}\binom{3k}{k}
x^k
 %\left(\frac{c^2}{(1-c)^3}\right)^k
\equiv
-3+3(1-4c)^{(q-1)/2}
\pmod{(c^q-c,p)},
\end{equation}
which take place in the polynomial ring $\Z[c]$, with $x=c^2/(1-c)^3$.
In particular, when evaluating those sums on a $p$-adic integer $c$ these congruences may be used in place of the more general Corollary~\ref{cor:3k_su_k},
as $c^p\equiv c\pmod{p}$ then.
In fact, the first of a set of four congruences proved in~\cite[Theorem~1.1]{Sun:sums_higher_Catalan} amounts to Equation~\eqref{eq:ZWS} evaluated on a $p$-adic integer $c$,
with $c\not\equiv 0,1,-2\pmod{p}$.
Although Equations~\eqref{eq:ZWS_short} and~\eqref{eq:ZWS}
give no information when $c=-2$,
the corresponding value for $x$ is also obtained for $c=1/4$,
where they give
$
\sum_{0<k<q/3}\binom{3k}{k}(4/27)^k
\equiv -2/3\pmod{p},
$
and
$
\sum_{0<k<q}\binom{3k}{k}(4/27)^k
\equiv -8/9\pmod{p}.
$
The latter congruence appeared in~\cite[Theorem~3.1]{Sun:sums_higher_Catalan}.

The fact that Equations~\eqref{eq:ZWS_short} and ~\eqref{eq:ZWS}
have the same right-hand side shows that the sums over the ranges $0<k<q/3$ and
$0<k<q$ are related in a simple way when $c\in\F_q$.
In particular, for $c\in\F_q\setminus\{ 1\}$
either sum vanishes if and only if the other one does.
Our next result determines when the sum over the short range vanishes (modulo $p$).

\begin{theorem}\label{thm:zero}
Let $p>3$ be a prime and let $q$ be a power of $p$, and
let $a\in\F_q$ with $a\neq 0, 1/9,4/27$.
Then the equality
$
\sum_{0<k<q/3}\binom{3k}{k}a^k
=0
$
holds if, and only if,
the polynomial $a(1-z)^3-z^2$ has three roots in $\F_q$.
\end{theorem}

The special case $q=p$ of Theorem~\ref{thm:zero} is in~\cite[Theorem~2.1]{SunZH:sums_(3k_k)},
under the additional assumption $a\neq 1/27$, which appears superfluous with our proof.
% Note that, according to Equations~\eqref{eq:ZWS_short} and~\eqref{eq:ZWS},
% the vanishing of the sums $\sum_{0<k<q/3}\binom{3k}{k}a^k$ and $\sum_{0<k<q}\binom{3k}{k}a^k$
% are equivalent conditions if a root $c$ of $a(1-z)^3-z^2$ satisfies $c\in\F_q\setminus\{1\}$.
Theorem~\ref{thm:zero} does not extend to the excluded case $a=1/9$.
In fact, according to Equation~\eqref{eq:3kk91}, which we will obtain by different means introduced
in Section~\ref{sec:Cardano},
when $q\equiv\pm 2\pmod{9}$ we have
$
\sum_{0<k<q/3}\binom{3k}{k}9^{-k}\equiv 0\pmod{p}
$.
However, according to Equation~\eqref{eq:3kk9},
when $q\equiv\pm 2\pmod{9}$ we also have
$
\sum_{0<k<q}\binom{3k}{k}9^{-k}\equiv -1\pmod{p}
$.
Consequently, the polynomial $(1-z)^3-9z^2$ has no roots in $\F_q$,
because if any such root $c$ existed then
according to Equations~\eqref{eq:ZWS_short} and~\eqref{eq:ZWS}
the sums on the shorter range would equal $1-c$ times the sum over the longer range.

\begin{proof}
Suppose first that all roots of cubic polynomial $a(1-z)^3-z^2$ belong to $\F_q$.
They are distinct because its discriminant $a(4-27a)$ is not zero.
Moreover, neither $1$ nor $-2$ is a root.
According to Equation~\eqref{eq:ZWS_short}, for each root $c\in\F_q$ of $a(1-z)^3-z^2$ we have
\[
\sum_{0<k<q/3}\binom{3k}{k}a^k
=
\frac{-3+3(1-4c)^{(q-1)/2}}{2(2+c)}
\in\left\{0,\frac{-3}{2+c}\right\},
\]
because $(1-4c)^{(q-1)/2}=\pm 1$.
Because the latter alternative can hold for at most one value of $c$,
we conclude that the former alternative holds, which is the desired conclusion.

In the opposite direction, suppose
$\sum_{0<k<q/3}\binom{3k}{k}a^k=0$,
and let $c$ satisfy
$a(1-c)^3-c^2=0$, with $c$ in the algebraic closure of $\F_q$.
Our goal is to how that $c^q=c$, which is equivalent to $c\in\F_q$.
The first congruence of Corollary~\ref{cor:3k_su_k} with $x=a$ yields
\[
2(2+c)
(1-c)^{q-1}
=
1+3(1-4c)^{(q-1)/2}
+2c^q,
\]
or, equivalently,
\[
(4+2c)
(1-c^q)
-(1+2c^q)(1-c)
=
3(1-4c)^{(q-1)/2}(1-c),
\]
which simplifies to
\[
1+c-2c^q
=
(1-4c)^{(q-1)/2}(1-c).
\]
% \[
% 2(2+c)
% (1-c^q)
% -(1+2c^q)(1-c)
% =
% 3(1-4c)^{(q-1)/2}(1-c).
% \]
Squaring both sides and then multiplying by $1-4c$ yields
\[
\bigl((1-c)-2(c^q-c)\bigr)^2
(1-4c)
=
(1-4c^q)(1-c)^2,
\]
which is equivalent to
\[
4(c^q-c)(1-c)^2
-4(c^q-c)(1-c)(1-4c)
+4(c^q-c)^2(1-4c)
=
0.
\]
Unless $c^q=c$, which is the desired conclusion,
we deduce
\[
(1-c)^2
-(1-c)(1-4c)
+(c^q-c)(1-4c)
=
0,
\]
whence
$1-c^q=(1-c)^2/(1-4c)$,
and
$c^q=-c(2+c)/(1-4c)$.
Because $c\neq 0,1$ we also find
$c^{q-1}=-(2+c)/(1-4c)$
and
$(1-c)^{q-1}=(1-c)/(1-4c)$.

At this point we use the information that $a\in\F_q^\ast$, which means $a^{q-1}=1$, and reads
$c^{2(q-1)}=(1-c)^{3(q-1)}$
in terms of $c$.
Substituting the expressions that we just found
for $c^{q-1}$ and $(1-c)^{q-1}$
we find
$(2+c)^2(1-4c)=(1-c)^3$.
Noting that $(2+c)^2(1-4c)=4(1-c)^3-27c^2$ we find $(1-c)^3=9c^2$,
in contrast with our hypothesis $a\neq 1/9$.
This contradiction concludes the proof.
\end{proof}

In the rest of this section we discuss some consequences of Theorem~\ref{thm:zero}.
If $a\in\F_q$ then $a(1-z)^3-z^2$, like any cubic polynomial in $\F_q[x]$,
has all its roots in $\F_{q^2}$ or $\F_{q^3}$, and hence splits into linear factors
over the extension field $\F_{q^6}$.
Therefore, as an example, when $a=1$ we find
\begin{equation}\label{eq:a=1}
\sum_{0<k<q/3}\binom{3k}{k}\equiv 0\pmod{p}
\end{equation}
for $p>3$ and $p\neq 23$, and $q$ a power of $p^6$.
This is the crucial case of~\cite[Theorem~1.4]{Sun:sums_higher_Catalan}, which was proved there in a more complicated way.
Of course the hypothesis that $q$ is a power of $p^6$ can be relaxed to
the polynomial $(1-z)^3-z^2$ splitting into linear factors over $\F_q$.

Similarly,
for $p>3$ and $p\neq 31$, and $q$ any power of $p^6$ we have
\begin{equation}\label{eq:a=-1}
\sum_{0<k<q/3}\binom{3k}{k}(-1)^k\equiv 0\pmod{p},
\end{equation}
Combining Equations~\eqref{eq:a=1} and~\eqref{eq:a=-1} we find
\begin{equation*}
\sum_{0<h<q/6}\binom{6h}{2h}\equiv
\sum_{0<h<q/6}\binom{6h-3}{2h-1}\equiv
0\pmod{p}
\end{equation*}
for $p>3$ and $p\not\in\{23,31\}$, and $q$ any power of $p^6$.

If
$a=4/(27+m^2)$ with $m\in\Q$, then
\[
\sum_{0<k<q/3}\binom{3k}{k}a^k\equiv 0\pmod{p},
\]
whenever $q$ is a power of $p^3$ and $a\in\Z_p$.
This is because the discriminant $a(1-4m)$ of the polynomial
$a(1-z)^3-z^2$ is then a perfect square (equal to $(am)^2$),
and hence all roots of the polynomial viewed modulo $p$ belong to $\F_{p^3}$.
The special case where $q=p$ is part of~\cite[Theorem~2.5]{SunZH:sums_(3k_k)}.

Theorem~\ref{thm:zero} can also be applied to algebraic integer values for $a$, such as $a=i$.
% , one of the complex square roots of $-1$.
With $p>3$, imposing $i^2\not\equiv (4/27)^2\pmod{p}$ amounts to $5\cdot 149\not\equiv 0\pmod{p}$.
Consequently, if $p>3$ and $p\not\in\{5,149\}$, the congruence
\begin{equation*}
\sum_{0<k<q/3}\binom{3k}{k}i^k\equiv 0\pmod{p},
\end{equation*}
holds for any power $q$ of $p^6$ if $p\equiv 1\pmod{4}$,
and for $q$ a power of $p^{12}$ if $p\equiv -1\pmod{4}$.
Together with Equations~\eqref{eq:a=1} and~\eqref{eq:a=-1},
under the same assumptions but including $p\not\in\{23,31\}$ we conclude
\begin{equation*}
\sum_{0<h<q/12}\binom{12h}{4h}\equiv
0\pmod{p}.
\end{equation*}

In a similar fashion, one may take $a=\pm\omega$, where $\omega=(-1+\sqrt{-3})/2$.
For example, taking $a=\omega$, and combining with Equation~\eqref{eq:a=1},
if $p>3$ and $p\not\in\{23, 853\}$ one concludes that
\begin{equation*}
\sum_{0<h<q/9}\binom{9h}{3h}\equiv
0\pmod{p}
\end{equation*}
holds for $q$ a power of $p^6$ if $p\equiv 1\pmod{3}$,
and for $q$ a power of $p^{12}$ if $p\equiv -1\pmod{3}$.

\section{A different approach to the cubic equation}\label{sec:Cardano}

Now we take a different approach to the series $y=\sum_{k=0}^{\infty}\binom{3k}{k}x^k$.
According to Equation~\eqref{eq:y_equation} it satisfies $(4-27x)y^3-3y-1=0$.
In principle one may obtain a closed form for this generating function by solving this equation through Cardano's formula.
However, such a closed form would involve taking both a square root and a cube root,
and this is not well suited to further manipulations we intend to do
in order to deduce a congruence modulo a prime for a truncated version of the series.

The discriminant of $(4-27x)y^3-3y-1$, viewed as a polynomial in $y$, equals $3^6\cdot x(4-27x)$.
We would like to substitute a rational function for $x$ in such a way that the discriminant becomes the square of a rational function.
The most elegant substitution appears to be $x=4s^2/\bigl(27(s^2-1))\bigr)$, which amounts to $s^2=-27x/(4-27x)$,
for which the discriminant becomes $-3\cdot (12s)^2/(s^2-1)^2$.
Note that the discriminant is only a square up to the factor $-3$,
but some occurrence of a square root of $-3$ is bound to turn up somewhere
with any other choice of a substitution,
as solving the cubic equation by radicals requires the presence of a primitive cube root of unity $(-1\pm\sqrt{-3})/2$ in the ground field.
Adopting that substitution the series $y$ acquires the following simple closed form.

\begin{lemma}\label{lemma:3k_s}
In the power series ring $\Q[[s]]$ we have
\[
2\sum_{k=0}^\infty\binom{3k}{k}
\left(\frac{4s^2}{27(s^2-1))}\right)^k
=
(1+s)^{2/3}(1-s)^{1/3}+(1-s)^{2/3}(1+s)^{1/3}.
\]
\end{lemma}

\begin{proof}
According to the case $r=3$ of Equation~\eqref{eq:y_equation},
which reads $(4-27x)y^3-3y-1=0$,
after applying the substitution
$x=4s^2/\bigl(27(s^2-1)\bigr)$
the formal series
\[
y_1(s):=\sum_{k=0}^\infty\binom{3k}{k}
\left(\frac{4s^2}{27(s^2-1))}\right)^k
\in\Q[[s]]
\]
is a root of the polynomial
\[
\frac{4}{1-s^2}\cdot y^3-3y-1\in\bigl(\Q[[s]]\bigr)[y].
\]
Because $4y^3-3y-1=(y-1)(2y+1)^2$, the series $y_1(s)$
is the only root of this polynomial having constant term $1$.
The series
\[
y_2(s):=\frac{1}{2}
(1-s^2)^{1/3}
\cdot
\bigl((1+s)^{1/3}+(1-s)^{1/3}\bigr)\in\Q[[s]]
\]
has constant term $1$ and is also root of the same polynomial, whence $y_1(s)=y_2(s)$ as claimed.
\end{proof}

Now we derive corresponding congruences for the finite sums.

\begin{theorem}\label{thm:pol_3k_s}
Set $x=4s^2/(27(s^2-1))$ in the polynomial ring $\Z[s]$.
Let $q$ be a power of the prime $p>3$,
and set $\varepsilon=\left(\frac{q}{3}\right)$, a Legendre symbol.
Thus, $\varepsilon=\pm1$ according to whether $q\equiv\pm1\pmod{3}$.
Then
\[
2(1-s^{2})^{(2q-3+\varepsilon)/6}
\sum_{0\le k<q/3}\binom{3k}{k}x^k
\equiv
(1+s)^{(2q+\varepsilon)/3}+(1-s)^{(2q+\varepsilon)/3}
\pmod{p},
\]
and
\begin{align*}
2&(1-s^2)^{(4q-3-\varepsilon)/6}
\sum_{0\le k<2q/3}\binom{3k}{k}
x^k
\\&\equiv
(1+s)^{(q-\varepsilon)/3}(1-s^q/3)+(1-s)^{(q-\varepsilon)/3}(1+s^q/3)
\pmod{p}.
\end{align*}
\end{theorem}

From the two congruences of Theorem~\ref{thm:pol_3k_s} one obtains the polynomial congruence
\begin{align*}
3&(1-s^2)^{(4q-3-\varepsilon)/6}
s^{-q}
\sum_{q/2\le k<2q/3}\binom{3k}{k}
x^k
\\&\equiv
(1+s)^{(q-\varepsilon)/3}-(1-s)^{(q-\varepsilon)/3}
\pmod{p}.
\end{align*}

\begin{proof}
Starting from the identity of power series in Lemma~\ref{lemma:3k_s} we produce polynomial congruences in the usual way.
We start with the shorter range, noting that $\sigma=(2q-3+\varepsilon)/6$ is the largest integer which is less than $q/3$.

On the one hand we have
\begin{align*}
2&(1-s^{2})^{\sigma}
\sum_{k=0}^\infty\binom{3k}{k}
x^k
\\&\equiv
2(1-s^{2})^{\sigma}
\sum_{0\le k<q/3}\binom{3k}{k}
\left(\frac{4s^2}{27(s^2-1))}\right)^k
\pmod{(s^q,p)}
\\&=
2\sum_{0\le k<q/3}\binom{3k}{k}
(-4s^2/27)^k(1-s^2)^{\sigma-k}.
\end{align*}
This final expression is a polynomial in $s$, of degree at most $2\sigma$, which is less than $q$.
On the other hand,
because $(1\pm s)^{q/3}\equiv 1\pmod{(s^q,p)}$,\
 %for $q\equiv 1\pmod{3}$, whence $\sigma=(q-1)/3$,
for $q\equiv 1\pmod{3}$ we have
\begin{align*}
2&(1-s^{2})^{(q-1)/3}
\sum_{k=0}^\infty\binom{3k}{k}
x^k
\\&=
(1+s)^{(q+1)/3}(1-s)^{q/3}+(1-s)^{(q+1)/3}(1+s)^{q/3}
\\&\equiv
(1+s)^{(2q+1)/3}+(1-s)^{(2q+1)/3}
\pmod{(s^q,p)}
\end{align*}
and for $q\equiv -1\pmod{3}$ we have
\begin{align*}
2&(1-s^{2})^{(q-2)/3}
\sum_{k=0}^\infty\binom{3k}{k}
x^k
\\&=
(1+s)^{q/3}(1-s)^{(q-1)/3}+(1-s)^{q/3}(1+s)^{(q-1)/3}
\\&\equiv
(1-s)^{(2q-1)/3}+(1+s)^{(2q-1)/3}
\pmod{(s^q,p)}.
\end{align*}

Now we prove the congruence over the longer range $0\le k<2q/3$.
Note that $q-1-\sigma=(4q-3-\varepsilon)/6$ is the largest integer which is less than $2q/3$.

On the one hand we have
\begin{align*}
2&(1-s^{2})^{q-1-\sigma}
\sum_{k=0}^\infty\binom{3k}{k}
x^k
\\&\equiv
2(1-s^{2})^{q-1-\sigma}
\sum_{0\le k<2q/3}\binom{3k}{k}
\left(\frac{4s^2}{27(s^2-1))}\right)^k
\pmod{(s^{2q},p)}
\\&=
2\sum_{0\le k<2q/3}\binom{3k}{k}
(-4s^2/27)^k(1-s^2)^{{q-1-\sigma}-k}.
\end{align*}
This last expression is a polynomial in $s$, of degree not exceeding $2q-2-2\sigma$, which is less than $2q$.
On the other hand,
because $(1\pm s)^{q/3}\equiv 1\pm s^q/3\pmod{(s^{2q},p)}$,\
for $q\equiv 1\pmod{3}$ we have
\begin{align*}
2&(1-s^{2})^{(2q-2)/3}
\sum_{k=0}^\infty\binom{3k}{k}
x^k
\\&=
(1-s^2)^{q/3}(1+s)^{q/3}(1-s)^{(q-1)/3}
+(1-s^2)^{q/3}(1-s)^{q/3}(1+s)^{(q-1)/3}
\\&\equiv
(1-s)^{(q-1)/3}(1+s^q/3)
+(1+s)^{(q-1)/3}(1-s^q/3)
\pmod{(s^{2q},p)}.
\end{align*}
and for $q\equiv -1\pmod{3}$ we have
\begin{align*}
2&(1-s^{2})^{(2q-1)/3}
\sum_{k=0}^\infty\binom{3k}{k}
x^k
\\&=
(1-s^2)^{q/3}(1+s)^{(q+1)/3}(1-s)^{q/3}
+(1-s^2)^{q/3}(1-s)^{(q+1)/3}(1+s)^{q/3}
\\&\equiv
(1+s)^{(q+1)/3}(1-s^q/3)+(1-s)^{(q+1)/3}(1+s^q/3)
\pmod{(s^{2q},p)}.
\end{align*}
This concludes our proof.
\end{proof}

\section{Some numerical applications of Theorem~\ref{thm:pol_3k_s}}\label{sec:numerical}

In this final section we give several numerical applications of Theorem~\ref{thm:pol_3k_s}
by assigning some interesting values to $s$.
Recall that $x=4s^2/(27(s^2-1))$.
To simplify notation, all unadorned congruences in this section are meant modulo $p$, with $p>3$.

For $s=3$ the two congruences of Theorem~\ref{thm:pol_3k_s} give
\begin{equation*}
\label{eq:3kk61}
\sum_{0\leq k<q/3}\binom{3k}{k}\frac{1}{6^k}
\equiv
\begin{cases}
2^{(2q-1)/3}-2^{(q+1)/3}&\text{if $q\equiv -1\pmod{3}$,}\\
2^{(q+2)/3}-2^{(2q-2)/3}&\text{if $q\equiv 1\pmod{3}$,}
\end{cases}
\end{equation*}
and
\begin{equation*}\label{eq:3kk62}
\sum_{k=0}^{q-1}\binom{3k}{k}\frac{1}{6^k}
\equiv
\begin{cases}
-2^{(q-2)/3}&\text{if $q\equiv -1\pmod{3}$,}\\
2^{(q-1)/3}&\text{if $q\equiv 1\pmod{3}$,}
\end{cases}
\end{equation*}
the second of which is one of the assertions of~\cite[Theorem~1.2]{Sun:sums_higher_Catalan}.

For $s=i\sqrt{3}=1+2\omega=-1-2\omega^{-1}$, where $\omega=\exp(2\pi i/3)$, we have $s^2=-3$
and $(1\pm s)^3=-8=(-2)^3$.
Write $q\equiv b\pmod{9}$, with $b\in\{\pm 1,\pm 2,\pm 4\}$
(as we are assuming $p>3$).
When $q\equiv -1\pmod{3}$, that is, $b\in\{-1,2,-4\}$, we have
\begin{align*}
\sum_{0\leq k<q/3}\binom{3k}{k}\frac{1}{9^k}
&\equiv
2^{-(2q-1)/3}
\cdot\bigl(-2\omega^{-1})^{(2q-1)/3}+(-2\omega)^{(2q-1)/3}\bigr)
\\&\equiv
-\omega^{-(2b-1)/3}-\omega^{(2b-1)/3}
\pmod{p},
\end{align*}
which is congruent to $1$, $1$ or $-2$ according as $b=-1$, $b=2$ or $b=-4$.
Together with a similar calculation for the case $q\equiv 1\pmod{3}$, we obtain
\begin{equation}\label{eq:3kk91}
\sum_{0\leq k<q/3}\binom{3k}{k}
\frac{1}{9^k}
\equiv
\begin{cases}
1&\text{if $q\equiv\pm 1\pmod{9}$,}\\
1&\text{if $q\equiv\pm 2\pmod{9}$,}\\
-2&\text{if $q\equiv\pm 4\pmod{9}$.}
\end{cases}
\end{equation}
Similarly, we find
\begin{equation}\label{eq:3kk9}
\sum_{k=0}^{q-1}\binom{3k}{k}\frac{1}{9^k}
\equiv
\begin{cases}
1&\text{if $q\equiv\pm 1\pmod{9}$,}\\
0&\text{if $q\equiv\pm 2\pmod{9}$,}\\
-1&\text{if $q\equiv\pm 4\pmod{9}$,}
\end{cases}
\end{equation}
as in~\cite[Theorem~1.5]{Sun:sums_higher_Catalan}.
Note that according to Lemma~\ref{lemma:3k_s} we have
$\sum_{k=0}^{\infty}\binom{3k}{k}9^{-k}
=\exp(i\pi/9)+\exp(-i\pi/9)
=2\cos(\pi/9)$.

For $s=1/\sqrt{5}$ we have $(1\pm s)=\pm2\phi_{\pm}/\sqrt{5}$
with $\phi_{\pm} = (1 \pm \sqrt{5} )/2$.
Letting $\varepsilon=\leg{q}{3}$ as in Theorem~\ref{thm:pol_3k_s}, we find
\[
\sum_{0\leq k<q/3}\binom{3k}{k}
\left(-\frac{1}{27}\right)^k
\equiv
\frac{(\phi_+)^{(2q+\varepsilon)/3}-(\phi_-)^{(2q+\varepsilon)/3}}{\sqrt{5}}
=
F_{(2q+\varepsilon)/3}.
%\equiv \leg{q}{5}F_{(q-\varepsilon)/3-\leg{q}{5}}.
\]
Note that
$F_{(2q+\varepsilon)/3}
\equiv \leg{q}{5}F_{(q-\varepsilon)/3-\leg{q}{5}}
\pmod{p}
$
because $2\phi_{\pm}^p=1\pm\leg{p}{5}\sqrt{5}$, see \cite[p.144]{MatTau:polylog}, for example.
Taking this into account we recover the congruence in~\cite[Corollary 3.1]{SunZH:lucas}.
In a similar way we obtain
\[
\sum_{q/2<k<2q/3}\binom{3k}{k}
\left(-\frac{1}{27}\right)^k
\equiv
\frac{(\phi_+)^{(q-\varepsilon)/3}-(\phi_-)^{(q-\varepsilon)/3}}{3\sqrt{5}}
=
\frac{F_{(q-\varepsilon)/3}}{3}.
\]
In this case the corresponding power series converges, and according to Lemma~\ref{lemma:3k_s}
\[
\sum_{k=0}^{\infty}\binom{3k}{k}\left(-\frac{1}{27}\right)^k
=\frac{(\phi_+)^{1/3}+(\phi_-)^{1/3}}{\sqrt{5}}
=\frac{2\cosh(\ln(\phi_+)/3)}{\sqrt{5}}.
\]

By setting
$s=2/\sqrt{5},3/\sqrt{5},i/\sqrt{3},i$
in Theorem~\ref{thm:pol_3k_s}
one obtains similar congruences for
$x=-16/27,1/3,1/27, 2/27$, respectively.

\bibliography{References}

\end{document}